\documentclass[12pt]{amsart}

\newtheorem{theorem}{Theorem}[section]

\theoremstyle{definition}

\newtheorem{corollary}{Corollary}[section]

\theoremstyle{remark}

\numberwithin{equation}{section}
\newcommand{\abs}[1]{\lvert#1\rvert}

\begin{document}
\title[Totally real submanifolds of $(LCS)_n$-Manifolds]{Totally real submanifolds of $(LCS)_n$-Manifolds.}
\author[S. K. Hui  and T. Pal]{SHYAMAL KUMAR HUI and TANUMOY PAL}
.
\begin{abstract}
The present paper deals with the study of totally real submanifolds and $\textit{C}$-totally real submanifolds of $(LCS)_n$-manifolds with
respect to Levi-Civita connection as well as quarter symmetric metric connection. It is proved that scalar curvature of $\textit{C}$-totally real submanifolds of $(LCS)_n$-manifold with respect to both the said connections are same.
\end{abstract}
\maketitle
\section{Introduction}
\indent As a generalization of LP-Sasakian manifold, recently Shaikh \cite{SHAIKH2} introduced the
notion of Lorentzian concircular structure manifolds (briefly, $(LCS)_n$-manifolds) with an example.
Such manifolds has many applications in the general theory of relativity and cosmology
(\cite{SHAIKH3}, \cite{SHAIKH4}).\\
\indent The notion of semisymmetric linear connection on a smooth manifold is introduced by Friedmann and Schouten \cite{FRSC}.
Then Hayden \cite{HAYDEN} introduced the idea of metric connection with torsion on a Riemannian manifold.
Thereafter Yano \cite{YANO} studied semisymmetric metric connection on a Riemannian manifold systematically. As a generalization of
semisymmetric connection, Golab \cite{GOLAB} introduced the idea
of quarter symmetric linear connection on smooth manifolds. A linear connection $\overline{\nabla}$ in an
$n$-dimensional smooth manifold $\widetilde{M}$ is said to be a quarter symmetric connection \cite{GOLAB}
if its torsion tensor $T$ is of the form
\begin{eqnarray}\label{1.1}
T(X,Y)&=&\overline{\nabla}_XY-\overline{\nabla}_YX-[X,Y]\\
\nonumber&&=\eta(Y)\phi X-\eta(X)\phi Y,
\end{eqnarray}
where $\eta$ is an 1-form and $\phi$ is a tensor of type $(1,1)$.
In particular, if $\phi X=X$ then the quarter symmetric connection reduces to semisymmetric connection.
Further if the quarter symmetric connection $\overline{\nabla}$ satisfies the condition $(\overline{\nabla}_Xg)(Y,Z)=0$
for all $X,\ Y,\ Z\in \chi(\widetilde{M})$, the set of all smooth vector fields on $\widetilde{M}$, then $\overline{\nabla}$
is said to be a quarter symmetric metric connection.\\
\indent Due to important applications in applied mathematics and theoretical physics, the geometry of submanifolds has become a subject of growing interest.
Analogous to almost Hermitian manifolds, the invariant and anti-invariant submanifolds are depend on the behaviour of almost contact metric structure
$\phi$. A submanifold of a contact metric manifold $\widetilde{M}$ is said to be anti-invariant if for any $X$ tangent to $M$, $\phi X$ is normal to $M$,
 i.e., $\phi (TM)\subset T^\bot M$ at every point of $M$, where $T^\bot M$ is the normal bundle of $M$. So, if a submanifold $M$ of a contact metric manifold $\widetilde{M}$ is normal to the structure vector field $\xi$, then it is anti-invariant. A submanifold $M$ in a contact metric manifold $\widetilde{M}$ is called a $\textit{C}$-totally real submanifold if every tangent vector of $M$ belongs to the contact distribution \cite{yama}. Thus a submanifold $M$ in a contact metric manifold is a $\textit{C}$-totally real submanifold if $\xi$ is normal to $M$. Consequently $\textit{C}$-totally real submanifolds in a contact metric manifold are anti-invariant, as they are normal to $\xi$. Recently Hui et al. (\cite{ATEHUI}, \cite{HUIAT}, \cite{HUIATPAL}, \cite{SHAIKH9}) 
 studied submanifolds of $(LCS)_{n}$-manifolds. The present paper deals with the totally real submanifolds and $\textit{C}$-totally real submanifolds of  $(LCS)_n$-manifolds with respect to Levi-Civita as well as quarter symmetric metric connection. It is shown that the scalar curvature of a $\textit{C}$-totally real submanifold of $(LCS)_n$-manifolds with respect to Levi-Civita connection and quarter symmetric metric connection are same. However in case of totally real submanifolds of $(LCS)_n$-manifolds with respect to Levi-Civita connection and quarter symmetric metric connection they are different. An inequality for the square length of the shape operator in case of totally real submanifold of $(LCS)_n$-manifold is derived. The equality case is also considered.
\section{preliminaries}
Let $\widetilde{M}$ be an $n$-dimensional Lorentzian manifold \cite{NIL} admitting a unit
timelike concircular vector field $\xi$, called the characteristic
vector field of the manifold. Then we have
\begin{equation}
\label{2.1}
g(\xi, \xi)=-1.
\end{equation}
Since $\xi$ is a unit concircular vector field, it follows that
there exists a non-zero 1-form $\eta$ such that for
\begin{equation}
\label{2.2}
g(X,\xi)=\eta(X),
\end{equation}
the equation of the following form holds \cite{15}
\begin{equation}
\label{2.3}
(\widetilde\nabla _{X}\eta)(Y)=\alpha \{g(X,Y)+\eta(X)\eta(Y)\},
\ \ \ \alpha\neq 0,
\end{equation}
\begin{equation}
\label{2.4}
\widetilde\nabla _{X}\xi = \alpha \{X +\eta(X)\xi\}, \ \ \ \alpha\neq 0,
\end{equation}
for all vector fields $X$, $Y$, where $\widetilde{\nabla}$ denotes the
operator of covariant differentiation with respect to the Lorentzian
metric $g$ and $\alpha$ is a non-zero scalar function satisfies
\begin{equation}
\label{2.5}
{\widetilde\nabla}_{X}\alpha = (X\alpha) = d\alpha(X) = \rho\eta(X),
\end{equation}
$\rho$ being a certain scalar function given by $\rho=-(\xi\alpha)$.
Let us take
\begin{equation}
\label{2.6}
\phi X=\frac{1}{\alpha}\widetilde\nabla_{X}\xi,
\end{equation}
then from (\ref{2.4}) and (\ref{2.6}) we have
\begin{equation}
\label{2.7} \phi X = X+\eta(X)\xi,
\end{equation}
\begin{equation}
\label{2.8}
g(\phi X,Y) = g(X,\phi Y),
\end{equation}
from which it follows that $\phi$ is a symmetric (1,1) tensor and
called the structure tensor of the manifold. Thus the Lorentzian
manifold $\widetilde{M}$ together with the unit timelike concircular vector
field $\xi$, its associated 1-form $\eta$ and an (1,1) tensor field
$\phi$ is said to be a Lorentzian concircular structure manifold
(briefly, $(LCS)_{n}$-manifold), \cite{SHAIKH2}. Especially, if we take
$\alpha=1$, then we can obtain the LP-Sasakian structure of
Matsumoto \cite{8}. In a $(LCS)_{n}$-manifold $(n>2)$, the following
relations hold \cite{SHAIKH2}:
\begin{equation}
\label{2.9}
\eta(\xi)=-1,\ \ \phi \xi=0,\ \ \ \eta(\phi X)=0,\ \ \
g(\phi X, \phi Y)= g(X,Y)+\eta(X)\eta(Y),
\end{equation}
\begin{equation}
\label{2.10}
\phi^2 X= X+\eta(X)\xi,
\end{equation}
\begin{equation}
\label{2.11}
\widetilde{R}(X,Y)Z =\phi \widetilde{R}(X,Y)Z +(\alpha^{2}-\rho)\{g(Y,Z)\eta(X)-g(X,Z)\eta(Y)\}\xi
\end{equation}
for all $X,\ Y,\ Z\in\Gamma(T\widetilde{M})$ Using (\ref{2.8}) in (\ref{2.11}) we get,
\begin{equation}\label{2.12}
\widetilde{R}(X,Y,Z,W)=\widetilde{R}(X,Y,Z,\phi W)+(\alpha^{2}-\rho)\{g(Y,Z)\eta(X)-g(X,Z)\eta(Y)\}\eta(W)
\end{equation}
\indent Let $M$ be a submanifold of dimension $m$ of a $(LCS)_n$-manifold $\widetilde{M}$ $(m<n)$ with induced
metric $g$. Also let $\nabla$ and $\nabla^{\perp}$ be the induced
connection on the tangent bundle $TM$ and the normal bundle
$T^{\perp}M$ of $M$ respectively. Then the Gauss and Weingarten
formulae are given by
\begin{equation}\label{2.13}
\widetilde{\nabla}_{X}Y = \nabla_{X}Y + h(X,Y)
\end{equation}
and
\begin{equation}\label{2.14}
\widetilde{\nabla}_{X}V = -A_{V}X + \nabla^{\perp}_{X}V
\end{equation}
for all $X,\ Y \in\Gamma(TM)$ and $V\in\Gamma(T^{\perp}M)$, where $h$
and $A_V$ are second fundamental form and the shape operator
(corresponding to the normal vector field $V$) respectively for the
immersion of $M$ into $\widetilde{M}$ and they are related by \cite{YANO3}
\begin{equation}\label{2.15}
g(h(X,Y),V) = g(A_{V}X,Y)
\end{equation}
for any $X,\ Y \in\Gamma(TM)$ and $V\in\Gamma(T^{\perp}M)$. And the equation of Gauss is given by
\begin{equation}\label{2.16}
\widetilde{R}(X,Y,Z,W)=R(X,Y,Z,W)+g(h(X,Z),h(Y,W))-g(h(X,W),h(Y,Z))
\end{equation}
for any vectors $X,\ Y,\ Z,\ W$ tangent to $M$.\\
\indent Let $\{e_i:i=1,2,\cdots ,n\}$ be an orthonormal basis of the tangent space $\widetilde{M}$ such that, refracting to $M^m$,
$ \{e_1,e_2,\cdots ,e_m\}$ is the orthonormal basis to the tangent space $T_xM$ with respect to induced connection.\\
 We write $$h_{ij}^r=g(h(e_i,e_j),e_r).$$\\ Then the square length of $h$ is $$\abs{\abs{h}}^2=\sum_{i,j=1}^{m}g(h(e_i,e_j),h(e_i,e_j))$$\\
 and the mean curvature of $M$ associated to $\nabla$ is $$H=\frac{1}{m}\sum_{i=1}^{m}h(e_i,e_i).$$
The quarter symmetric metric connection $\overline{\widetilde{\nabla}}$ and Riemannian connection $\widetilde{\nabla}$ on a
$(LCS)_n$-manifold $\widetilde{M}$ are related by \cite{HUI}
\begin{equation}\label{2.17}
\overline{\widetilde{\nabla}}_XY=\widetilde{\nabla}_XY+ \eta(Y)\phi X-g(\phi X,Y)\xi.
\end{equation}
If $\overline{\widetilde{R}}$ and $\widetilde{R}$ are the curvature tensors of a $(LCS)_n$-manifold $\widetilde{M}$ with respect to
quarter symmetric metric connection $\overline{\widetilde{\nabla}}$ and Riemannian connection $\widetilde{\nabla}$, then
\begin{eqnarray}\label{2.18}
  \overline{\widetilde{R}}(X,Y,Z,W)&=& \widetilde{R}(X,Y,Z,W)+(2\alpha-1)[g(\phi X,Z)g(\phi Y,W)  \\
 \nonumber &&-g(\phi Y,Z)g(\phi X,W)]+\alpha [\eta(Y)g(X,W) \\
\nonumber &&-\eta(X)g(Y,W)]\eta(Z) +\alpha [ g(Y,Z)\eta(X)\\
\nonumber&&-g(X,Z)\eta(Y)]\eta(W)
\end{eqnarray}
for all $X,\ Y,\ Z\ \text{and}\ W$ on $\chi (\widetilde{M})$.\\
Let $L$ be a $k$-plane section of $T_xM$ and $X$ be a unit  vector in $L$. We choose an orthonormal basis $\{e_1,e_2,\cdots ,e_k\}$
of $L$ such that $e_1=X$.
\indent Then the Ricci curvature $Ric_L$ of $L$ at $X$ is defined by
\begin{equation}\label{2.19}
Ric_L(X)=K_{12}+K_{13}+\cdots+K_{1k},
\end{equation}
where $K_{ij}$ denotes the sectional curvature of the $2$-plane section spanned by $e_i, e_j$. Such a curvature is called a $k$-Ricci curvature.\\
\indent The scalar curvature $\tau$ of the $k$-plane section $L$ is given by
\begin{equation}\label{2.20}
  \tau(L)=\displaystyle\sum_{i\leq i<j\leq k}K_{ij}.
\end{equation}
\indent For each integer $k$, $2\leq k\leq n$, the Riemannian invariant $\Theta_k$ on an $n$-dimensional Riemannian manifold
$M$ is defined by
\begin{equation}\label{2.21}
\Theta_k(x)=\frac{1}{k-1}\displaystyle\inf_{L.X}Ric_L(X),\ \ \ \ x\in M,
\end{equation}
where $L$ runs over all $k$-plane sections in $T_xM$ and $X$ runs over all unit vectors in $L$.
indent The relative null space for a submanifold $M$ of a Riemannian manifold at a point $x\in M$ is defined by
\begin{equation}\label{2.22}
\emph{N}_x=\{X\in T_xM| h(X,Y)=0,Y\in T_xM\}.
\end{equation}
\section{Main results}
This section deals with the study of totally real submanifolds of $(LCS)_n$-manifolds with respect to Levi-Civita
 as well as quarter symmetric metric connection. We prove the following:
\begin{theorem}
  Let $M$ be a toatally real submanifold of dimension $m$ $(m<n)$ of a $(LCS)_n$-manifold $\widetilde{M}$. Then
  \begin{equation}\label{3.1}
  m^2\abs{\abs{H}}^2=2\tau +\abs{\abs{h}}^2+(m-1)(\alpha^2-\rho),
  \end{equation}
  where $\tau$ is the scalar curvature of $M$.
\end{theorem}
\begin{proof}
\indent Let $M$ be a totally real submanifold of a $(LCS)_n$-manifold $\widetilde{M}$.
Now from (\ref{2.12}) and (\ref{2.16}), we get
\begin{eqnarray}
 \label{3.2} R(X,Y,Z,W) &=& \widetilde{R}(X,Y,Z,\phi W)+(\alpha^2-\rho)\{g(Y,Z)\eta(X) \\
  \nonumber&&-g(X,Z)\eta(Y)\}\eta(W)+g(h(X,W),h(Y,Z))\\
  \nonumber&&-g(h(X,Z),h(Y,W))
\end{eqnarray}
for any $X,\ Y,\ Z,\ W\in \Gamma(TM)$.\\
Since $M$ is totally real submanifold i.e., anti-invariant so $$\widetilde{R}(X,Y,Z,\phi W)= g(\widetilde{R}(X,Y)Z,\phi W)=0$$ and hence
 (\ref{3.2}) yields
\begin{eqnarray}\label{3.3}
  R(X,Y,Z,W) &=& (\alpha^2-\rho)\{g(Y,Z)\eta(X)-g(X,Z)\eta(Y)\}\eta(W) \\
  \nonumber&&+g(h(X,W),h(Y,Z))-g(h(X,Z),h(Y,W))
\end{eqnarray}
for any $X,\ Y,\ Z,\ W\in \Gamma(TM)$.
Putting $X=W=e_i$ and $Y=Z=e_j$ in (\ref{3.3}) and taking summation over $1\leq i<j\leq m$, we get
\begin{eqnarray*}
  \displaystyle\sum_{1\leq i<j\leq m}R(e_i,e_j,e_j,e_i) &=& (\alpha^2-\rho)\displaystyle\sum_{1\leq i<j\leq m}[g(e_j,e_j)\eta(e_i)\eta(e_i)
-g(e_i,e_j)\eta(e_j)\eta(_j)]\\
&& + \displaystyle\sum_{1\leq i<j\leq m} g(h(e_i,e_i),h(e_j,e_j))- \displaystyle\sum_{1\leq i<j\leq m} g(h(e_i,e_j),h(e_j,e_i))
\end{eqnarray*}
i.e.,
\begin{equation}\label{3.4}
2\tau=-(m-1)(\alpha^2-\rho)+m^2\abs{\abs{H}}^2-\abs{\abs{h}}^2,
\end{equation}
which implies (\ref{3.1}).
\end{proof}
\begin{corollary}
  Let $M$ be a $\textit{C}$-totally real submanifold of dimension $m$ $(m<n)$ of a $(LCS)_n$-manifold $\widetilde{M}$. Then
  \begin{equation*}\label{3.1}
  m^2\abs{\abs{H}}^2=2\tau +\abs{\abs{h}}^2,
  \end{equation*}
  where $\tau$ is the scalar curvature of $M$.
\end{corollary}
\begin{proof}
\indent In a $\textit{C}$-totally real submanifold, it is known that $\eta(X)=0$ for all $X\in \Gamma(TM)$ and $\xi\in T^\bot M$.
 Then (\ref{3.3}) yields
 \begin{eqnarray*}
   R(X,Y,Z,W) &=& g(h(X,W),h(Y,Z))-g(h(X,Z),h(Y,W)),
 \end{eqnarray*}
 from which as similar in above we can prove that $ m^2\abs{\abs{H}}^2=2\tau +\abs{\abs{h}}^2$.
\end{proof}
\indent Now  let $M$ be a submanifold of dimension $m$ $(m<n)$ of a $(LCS)_n$-manifold $\widetilde{M}$ with respect to quarter
 symmetric metric connection $\overline{\widetilde{\nabla}}$  and $\overline{\nabla}$ be the induced connection of $M$
 associated to the quarter symmetric metric connection. Also let $\overline{h}$ be the second fundamental form of $M$
with respect to $\overline{\nabla}$. Then the Gauss formula can be written as
\begin{equation}\label{3.5}
\overline{\widetilde{\nabla}}_XY=\overline{\nabla}_XY+\overline{h}(X,Y)
\end{equation}
and hence by virtue of (\ref{2.13}) and (\ref{2.17}) we get
\begin{eqnarray}\label{3.6}
 \overline{\nabla}_XY+\overline{h}(X,Y) &=& \nabla_XY+h(X,Y)+\eta(Y)\phi X-g(\phi X,Y)\xi
\end{eqnarray}
If $M$ is totally real submanifold of $\widetilde{M}$ then $\phi X\in T^\bot M$ for any $X\in TM$ and hence $g(\phi X,Y)=0$ for $X,\ Y\in TM$.
So, equating the normal part from (\ref{3.6}) we get
\begin{equation}\label{3.7}
\overline{h}(X,Y)=h(X,Y)+\eta(Y)\phi X.
\end{equation}
Further, if $M$ is $\textit{C}$-totally real submanifold of $\widetilde{M}$ then $\xi \in T^\bot M$ and hence $\eta (Y)=0$ for all $Y\in TM$.
So, (\ref{3.7}) yields
\begin{equation}\label{3.8}
\overline{h}(X,Y)=h(X,Y).
\end{equation}
Let $U$ be a unit tangent vector at $x\in \widetilde{M}$ and $\{e_i:i=1,2,\cdots ,n\}$ be an orthonormal basis of the 
tangent space $\widetilde{M}$ such that $e_1=U$ refracting to $M^m$, $\{e_1,e_2,\cdots ,e_m\}$ is the orthonormal 
basis to the tangent space $T_xM$ with respect to induced quarter symmetric metric connection. Then we have the following:
\begin{theorem}
 Let $M$ be a totally real submanifold of a $(LCS)_n$-manifold $\widetilde{M}$ with respect to quarter symmetric metric connection then
 \begin{equation}\label{3.9}
 m^2\abs{\abs{H}}^2=2\overline{\tau}+\abs{\abs{h}}^2+(2m-1)\alpha+m\alpha\eta^2(U),
 \end{equation}
 where $\overline{\tau}$ is the scalar curvature of $M$ with respect to induced connection associated to the quarter symmetric metric connection.
 \end{theorem}
 \begin{proof}
\indent In case of $(LCS)_n$-manifold $\widetilde{M}$ with respect to quarter symmetric metric connection, the relation (\ref{2.16}) becomes
\begin{eqnarray}\label{3.10}
 \overline{\widetilde{R}}(X,Y,Z,W) &=&\overline{R}(X,Y,Z,W)+g(\overline{h}(X,Z),\overline{h}(Y,W)) \\
\nonumber&&-g(\overline{h}(X,W),\overline{h}(Y,Z)).
\end{eqnarray}
In view of (\ref{2.7}) and (\ref{2.8}) (\ref{3.10}) yields
\begin{eqnarray}\label{3.11}
  \overline{R}(X,Y,Z,W) &=& \widetilde{R}(X,Y,Z,\phi W)+(\alpha^2-\rho)\{g(Y,Z)\eta(X) \\
  \nonumber&&-g(X,Z)\eta(Y)\}\eta(W)+(2\alpha-1)[g(\phi X,Z)g(\phi Y,W)  \\
 \nonumber &&-g(\phi Y,Z)g(\phi X,W)]+\alpha [\eta(Y)g(X,W) \\
\nonumber &&-\eta(X)g(Y,W)]\eta(Z)+\alpha [ g(Y,Z)\eta(X)\\
\nonumber&&-g(X,Z)\eta(Y)]\eta(W) \\
 \nonumber&&+g(\overline{h}(X,W),\overline{h}(Y,Z))-g(\overline{h}(X,Z),\overline{h}(Y,W)).
\end{eqnarray}
Since $M$ is totally real, therefore $g(\phi X,Y)=0$ for all $X,\ Y\in TM$ and (\ref{3.7}) holds. Thus (\ref{3.11}) becomes
\begin{eqnarray}\label{3.12}
 \overline{R}(X,Y,Z,W) &=&(\alpha^2-\rho)\{g(Y,Z)\eta(X)-g(X,Z)\eta(Y)\}\eta(W)\\
\nonumber&& +\alpha [\eta(Y)g(X,W)-\eta(X)g(Y,W)]\eta(Z)\\
\nonumber&&+\alpha [ g(Y,Z)\eta(X)-g(X,Z)\eta(Y)]\eta(W) \\
 \nonumber&&+g(h(X,W),h(Y,Z))-g(h(X,Z),h(Y,W))\\
 \nonumber&&-\eta(Z)g(h(X,W),\phi Y)-\eta(W)g(\phi X,h(Y,Z))\\
 \nonumber&&+\eta(Z)g(\phi X,h(Y,W))+\eta(W)g(h(X,Z),\phi Y).
\end{eqnarray}
Putting $X=W=e_i$ and $Y=Z=e_j$ in (\ref{3.12}) and taking summation over $1\leq i<j\leq m$ we get
\begin{eqnarray}\label{3.13}
2\overline{\tau} &=&-(m-1)(\alpha^2-\rho)-\alpha(1+\eta^2(U))m-\alpha(m-1)\\
\nonumber&&+m^2\abs{\abs{H}}^2-\abs{\abs{h}}^2,
\end{eqnarray}
from which (\ref{3.9}) follows.
 \end{proof}
 \begin{corollary}
 Let $M$ be a $\textit{C}$-totally real submanifold of a $(LCS)_n$-manifold $\widetilde{M}$ with respect to quarter symmetric metric connection. Then
 \begin{equation}\label{3.14}
 m^2\abs{\abs{H}}^2=2\overline{\tau}+\abs{\abs{h}}^2,
 \end{equation}
 where $\overline{\tau}$ is the scalar curvature of $M$ with respect to induced quarter symmetric metric connection.
 \end{corollary}
 \begin{proof}
 \indent If $M$ is $\textit{C}$-totally real submanifold then $\eta(Y)=0$ for all $Y\in TM$ and hence (\ref{3.12}) implies that
 \begin{equation}\label{3.15}
 \overline{R}(X,Y,Z,W)=g(h(X,Z),h(Y,W))-g(h(X,W),h(Y,Z))
 \end{equation}
 from which, by similar as above (\ref{3.14}) follows.
 \end{proof}
 From Corollary 3.1 and Corollary 3.2 we get $\tau=\overline{\tau}$ i.e., the scalar curvature of $\textit{C}$-totally 
 real submanifold of a $(LCS)_n$-manifold with respect to induced Levi-Civita connection and induced quarter 
 symmetric metric connection are identical. Thus we can state the following:
 \begin{theorem}
 Let $M$ be a $\textit{C}$-totally real submanifold of a $(LCS)_n$-manifold $\widetilde{M}$. Then the scalar curvature of $M$ with respect to induced Levi-Civita connection and induced quarter symmetric metric connection are same.
 \end{theorem}
 Next we prove the following:
 \begin{theorem}
 Let $\widetilde{M}$ be a $(LCS)_n$-manifold and $M$ be a totally real submanifold of $\widetilde{M}$ of dimension $m\ (m<n)$. Then\\
 (i) for each unit vector $X\in T_xM$,
 \begin{equation}\label{3.16}
4 Ric(X)\leq m^2\abs{\abs{H}}^2+2(\alpha^2-\rho)(m-2)+4(m-2)(\alpha^2-\rho)\eta^2(X) ;
 \end{equation}
 (ii) in case of H(x)=0, a unit tangent vector $X$ at $x$ satisfies the equality case of (\ref{3.16}) if and only if
 $X$ lies in the relative null space $\emph{N}_x$ at $x$.\\
 (iii) the equality case of (\ref{3.16}) holds identically for all unit tangent vectors at $x$ if and only if either $x$ is a totally geodesic point or
  $m=2$ and $x$ is a totally umbilical point.
 \end{theorem}
 \begin{proof}
 Let $X\in T_xM$ be a unit tangent vector at $x$. We choose an orthonormal basis $\{e_1,e_2,\cdots ,e_m,e_{m+1},\cdots ,e_n\}$
  such that $\{e_1,\cdots ,e_m\}$ are tangent to $M$ at $x$ and $e_1=X$. Then from (\ref{3.1}) we have
  \begin{eqnarray}
\nonumber    m^2\|H\|^2 &=& 2\tau +\displaystyle \sum_{r=m+1}^{n}\{(h_{11}^r)^2+(h_{22}^r+\cdots +h_{mm}^r)^2)\}\\
\nonumber&&-2\displaystyle\sum_{r=m+1}^{n}\sum_{21\leq i<j\leq n}h_{ii}^rh_{jj}^r+(m-1)(\alpha^2-\rho)\\
 \label{3.17} &=& 2\tau +\frac{1}{2}\displaystyle\sum_{r=m+1}^{n}\{(h_{11}^r+\cdots+h_{mm}^r)^2+(h_{11}^r-h_{22}^r-\cdots -h_{mm}^r)^2\}\\
\nonumber&&+2\displaystyle\sum_{r=m+1}^{n}\sum_{i<j}(h_{ij}^r)^2-2\displaystyle\sum_{r=m+1}^{n}\sum_{2\leq i<j\leq n}h_{ii}^rh_{jj}^r+(m-1)(\alpha^2-\rho).
  \end{eqnarray}
  From the equation of Gauss, we find
  \begin{equation*}
  K_{ij}=\displaystyle\sum_{r=m+1}^{n}[h_{ii}^rh_{jj}^r-(h_{ij}^r)^2]+(\alpha^2-\rho)\eta^2(e_i),
  \end{equation*}
   and consequently
   \begin{eqnarray}\label{3.18}
  \displaystyle\sum_{2\leq i<j\leq m}K_{ij}=\displaystyle\sum_{r=m+1}^{n}\sum_{2\leq i<j\leq m}[h_{ii}^rh_{jj}^r-(h_{ij}^r)^2]
  +(\alpha ^2-\rho)[m-2+\eta^2(X)].
   \end{eqnarray}
   Using (\ref{3.18}) in (\ref{3.17}) we get
   \begin{eqnarray}
   \label{3.19}
     m^2\abs{\abs{H}}^2 &\geq& 2\tau +\frac{m^2}{2}\abs{\abs{H}}^2+2\displaystyle\sum_{r=m+1}^{n}\sum_{j=2}^{m}(h_{1j}^r)^2-2\displaystyle\sum_{2\leq i<j\leq m}K_{ij} \\
    \nonumber&& -(\alpha^2-\rho)(m-3)-2(m-2)(\alpha^2-\rho)\eta^2(X).
   \end{eqnarray}
   Therefore,
   \begin{equation*}
      m^2\abs{\abs{H}}^2\frac{1}{2}\geq 2Ric(X)-(\alpha^2-\rho)(m-3)-2(m-2)(\alpha^2-\rho)\eta^2(X),
   \end{equation*}
   from which we get (\ref{3.16})\\
   \indent Let us assume that $H(x)=0$. Then the equality holds in (\ref{3.16}) if and only if
   \begin{equation*}
  h_{11}^r=h_{22}^r=\cdots=h_{1m}^r=0,\text{and} \ \ \ h_{11}^r=h_{22}^r+\cdots+h_{mm}^r,\ \ \  r\in \{m+1,\cdots,n\}.
  \end{equation*}

Then $h_{1j}^r=0$ for every $j\in \{1,\cdots m\},r\in\{m+1\cdots n\}$, i.e., $X\in \emph{N}_x$.\\
\indent (iii) The equality case of (\ref{3.16}) holds for every unit tangent vector at $x$ if and only if
\begin{equation*}
 h_{ij}^r=0, i\neq j\ \text{and}\  h_{11}^r+h_{22}^r+\cdots+h_{mm}^r-2h_{ii}^r=0.
\end{equation*}
We distinguish two cases:\\
(a) $m\neq 2$,then $x$ is a totally geodesic point;\\
(b)$m=2$, it follows that $x$ is a totally umbilical point.\\
The converse is trivial.
 \end{proof}
Next we obtain the following:
 \begin{theorem}
 Let $M$ be a totally real submanifold of dimension $m$ of a $(LCS)_n$ manifold $\widetilde{M}(m<n)$ . Then we have
 \begin{eqnarray}\label{4.1}
   \abs{\abs{H}}^2 \geq \frac{2\tau}{m(m-1)}+\frac{1}{m}(\alpha^2-\rho).
 \end{eqnarray}
 \end{theorem}
 \begin{proof}
   We choose an orthonormal basis $\{e_1,\cdots e_m, e_{m+1},\cdots e_n\}$ at $x$ such that $e_{m+1}$ is parallel to the mean curvature vector $H(x)$, and
   $e_1,\cdots e_m$ diagonalise the shape operator $A_{e_{m+1}}$. Then the shape operator takes the form
   \begin{eqnarray}\label{4.2}
   A_{e_{m+1}}=
   \left(
     \begin{array}{ccccc}
       a_1 & 0 & 0 & \cdots & 0 \\
       0 & a_2 & 0 & \cdots & 0 \\
       \vdots & \vdots & \vdots & \vdots & \vdots \\
       0 & 0 & 0 & \cdots & a_n \\
     \end{array}
   \right),
   \end{eqnarray}
   $A_{e_r}=(h_{ij}^r),\ \ i,j=1,\cdots,m;r=m+2,\cdots ,n,\ \ \ \ \ traceA_{e_r}=\displaystyle \sum_{i=1}^{m}h_{ii}^r=0$ \\
 and from (\ref{3.1}) we get
   \begin{equation}\label{4.3}
   m^2\abs{\abs{H}}^2=2\tau +\displaystyle\sum_{i=1}^{m}a_i^2+\displaystyle\sum_{r=m+2}^{n}\sum_{i,j=1}^{m}(h_{ij}^r)^2+(m-1)(\alpha^2-\rho).
   \end{equation}
   On the other hand, since
   \begin{equation}\label{4.4}
   0\leq\displaystyle\sum_{i<j}(a_i-a_j)^2=(m-1)\displaystyle\sum_{i}a_i^2-2\displaystyle\sum_{i<j}a_ia_j,
   \end{equation}
   we obtain
   \begin{equation}\label{4.5}
   m^2\abs{\abs{H}}^2=\left(\displaystyle\sum_{i=1}^{m}a_i\right)^2+2\displaystyle\sum_{i}a_i^2-2\displaystyle\sum_{i<j}a_ia_j
   \leq m\displaystyle\sum_{i=1}^{m}a_i^2,
   \end{equation}
   which implies that
   \begin{equation}\label{4.6}
   \displaystyle\sum_{i}a_i^2\geq m \|H\|^2 .
   \end{equation}
  In view of (\ref{4.6}), (\ref{4.3}) yields
   \begin{equation}\label{4.7}
   m^2\|H\|^2\geq 2\tau +m\|H\|^2+(m-1)(\alpha^2-\rho),
   \end{equation}
   which implies (\ref{4.1}).
 \end{proof}
 \begin{theorem}
 Let $M$ be a totally real submanifold of dimension $m$ of a $(LCS)_n$ manifold $\widetilde{M}(m<n)$ . Then for any integer
  $k$, $2\leq k\leq m$, and any point $x\in M$, we have
  \begin{equation}\label{4.8}
  \|H\|^2(x)\geq \Theta_k(x)+\frac{1}{m}(\alpha^2-\rho).
  \end{equation}
  \end{theorem}
  \begin{proof}
  Let $\{e_1,e_2,\cdots e_m\}$ be an orthonormal basis of $T_xM$. Denote by $L_{i_1,\cdots i_k}$ the $k$-plane section spanned by
  $e_{i_1},\cdots e_{i_k}$. Then we have \cite{CHEN}
  \begin{equation}\label{4.9}
  \tau(x)\geq \frac{m(m-1)}{2}\Theta_k(x).
  \end{equation}
  Using (\ref{4.9}) in (\ref{4.1}), (\ref{4.8}) follows.
  \end{proof}
\vspace{0.1in}
\noindent{\bf Acknowledgement:} The first author (S. K. Hui) gratefully acknowledges to
the SERB (Project No.: EMR/2015/002302), Govt. of India for financial assistance of the work.

\vspace{0.1in}
\noindent S. K. Hui\\
Department of Mathematics, The University of Burdwan, Golapbag, Burdwan -- 713104, West Bengal, India\\
E-mail: shyamal\_hui@yahoo.co.in, skhui@math.buruniv.ac.in\\

\noindent T. Pal\\
A. M. J. High School, Mankhamar, Bankura -- 722144, West Bengal, India\\
E-mail: tanumoypalmath@gmail.com
\end{document}